\documentclass{AIMS}
 \usepackage{amsmath}
 \usepackage{graphics} 
 \usepackage{epsfig} 

  \usepackage{hyperref} 

\def\currentvolume{X}
 \def\currentissue{X}
  \def\currentyear{200X}
   \def\currentmonth{XX}
    \def\ppages{X--XX}

\newtheorem{theorem}{Theorem}[section]
\newtheorem{corollary}[theorem]{Corollary}

\newtheorem{lemma}[theorem]{Lemma}

\theoremstyle{definition}
\newtheorem{definition}[theorem]{Definition}
\newtheorem{remark}[theorem]{Remark}

\title[Rate of convergence of periodic solutions]
      {On the rate of convergence of periodic solutions in perturbed autonomous systems as the perturbation vanishes }

\author[Oleg Makarenkov, Paolo Nistri]{}


\email{omakarenkov@math.vsu.ru; pnistri@dii.unisi.it}

\thanks{The first author is supported by the Grant BF6M10 of Russian Federation Ministry of Education and CRDF (US),
and by RFBR Grants 07-01-00035, 06-01-72552, 05-01-00100. The
second author is supported by GNAMPA and the national research
project PRIN ``Control, Optimization and Stability of Nonlinear
Systems: Geometric and Topological Methods''.}

\newcommand{\epsi}{\varepsilon}

\begin{document}
\maketitle

\centerline{\scshape Oleg Makarenkov}
\medskip
{\footnotesize
 \centerline{Research Institute of Mathematics, Voronezh State University}
  \centerline{Ul. Universitetskaja pl. 1}
   \centerline{394006, Voronezh, Russia}
} \medskip\centerline{\scshape Paolo Nistri}
\medskip
{\footnotesize
 \centerline{Dipartimento di Ingegneria dell' Informazione, Universit\`a di Siena}
  \centerline{Via Roma, 56}
   \centerline{53100, Siena, Italy}
}

\medskip

 \centerline{(Communicated by Aim Sciences)}
 \medskip

\begin{abstract}
We consider an autonomous system in $\mathbb{R}^n$ having a limit
cycle $ x_0$ of period $T>0$ which is nondegenerate in a suitable
sense. We then consider the perturbed system obtained by adding to
the autonomous system a $T$-periodic, not necessarily
differentiable, term whose amplitude tends to $0$ as a small
parameter $\varepsilon>0$ tends to $0.$ Assuming the existence of a
$T$-periodic solution $x_\varepsilon$ of the perturbed system and
its convergence to $ x_0$ as $\varepsilon\to 0$, the paper
establishes the existence of $\Delta_{\varepsilon}\to 0$ as
$\epsi\to 0$ such that $\|x_\epsi(t+\Delta_\epsi)-x_0(t)\|\le\epsi
M$ for some $M>0$ and any $\epsi>0$ sufficiently small.
 This paper completes the
work initiated by the authors in \cite{nach} and \cite{jmaa}.
Indeed, in \cite{nach} the existence of a family of $T$-periodic
solutions $x_\varepsilon$ of the perturbed system considered here
was proved. While in \cite{jmaa} for perturbed systems in
$\mathbb{R}^2$ the rate of convergence was investigated by means
of the method considered in this paper.

\end{abstract}

\section{Introduction}
Assume that the perturbed autonomous system
\begin{equation}\label{ps}
   \dot x=f(x)+\varepsilon g(t,x,\epsi), \qquad
   x\in\mathbb{R}^n,\;\; t\in\mathbb{R},
\end{equation}
possesses a family of $T$-periodic solutions
$\{x_\epsi\}_{\epsi\in(0,1]}$ such that
\begin{equation}\label{convp}
  x_\varepsilon(t)\to {x}_0(t)\mbox{\quad as\ }\varepsilon\to
  0
\end{equation}
uniformly with respect to $t\in\mathbb{R},$ where ${x}_0$ is a limit
cycle of period $T>0$ of the system
\begin{equation}\label{np}
   \dot x=f(x).
\end{equation}
The following system is an example of (\ref{ps}) having a family of
$2\pi$-periodic solutions $\{x_\epsi\}_{\epsi\in(0,1]}$ satisfying
(\ref{convp}), where the $2\pi$-periodic limit cycle ${x}_0$ is
represented by the circumference centered at the origin with radius
$1$.
$$
\begin{array}{l}
  \dot
  x_1=x_2-x_1\left(x_1^2+x_2^2-(1+\epsi)^2\right)+\epsi\left((1+\epsi)\sin(t-\sqrt{\epsi})-x_1\right),\\
  \dot x_2=-x_1-x_2\left(x_1^2+x_2^2-(1+\epsi)^2\right).
\end{array}
$$
In fact, as it is easy to see, for $\varepsilon>0$ it has the
$2\pi$-periodic solution
$$
  x_\epsi(t)=\left(\begin{array}{l}
     (1+\epsi)\sin(t-\sqrt{\epsi})\\
     (1+\epsi)\cos(t-\sqrt{\epsi})
     \end{array}\right)
$$
which, for any $t\in\mathbb{R},$ converges to $
{x}_0(t)=\left(\begin{array}{l}
     \sin t\\
     \cos t
     \end{array}\right)
$
when $\epsi\to 0.$ This example shows that the rate of
convergence in (\ref{convp}) can be less than $\varepsilon>0$,
indeed
$$
  \frac{\left\|x_\epsi(t)- {x}_0(t)\right\|}{\epsi}=
  \frac{1}{\epsi}\left\|\left(\begin{array}{l}
     \sin(t-\sqrt{\epsi})-\sin t\\
     \cos(t-\sqrt{\epsi})-\cos t
     \end{array}\right)+\varepsilon\left(\begin{array}{l}
     \sin(t-\sqrt{\epsi})\\
     \cos(t-\sqrt{\epsi})
     \end{array}\right)\right\|\to\infty
$$
as $\epsi\to 0.$

On the other hand the example also suggests that a suitable shift
in time in $x_\epsi$ gives convergence at the rate $\epsi.$ In
fact, we have that
\begin{equation}\label{sqr}
  \frac{\left\|x_\epsi(t+\sqrt{\epsi})- {x}_0(t)\right\|}{\epsi}=1\quad{\rm
  for\ any\ } t\in [0, 2\pi] \quad{\rm and\ any\ }\epsi>0.
\end{equation}

In this paper we show that the situation described by the above
example occurs in general. Namely we prove that, given a family of
$T$-periodic solutions $\{x_\epsi\}_{\epsi\in(0,1]}$ to (\ref{ps})
satisfying (\ref{convp}), it is always possible to find a suitable
family of shifts
$\left\{\Delta_\epsi\right\}_{\epsi>0}\subset\mathbb{R}$ satisfying
\begin{equation}\label{sqr1}
  \frac{\left\|x_\epsi(t+\Delta_\epsi)- {x}_0(t)\right\|}{\epsi}\le const\quad{\rm
  for\ any\ }t\in[0,T]{\rm\ and\ any\ }\epsi>0
\end{equation}
 provided that the limit cycle ${x}_0$ is
nondegenerate in the sense that the algebraic multiplicity of the
characteristic multiplier $+1$ of
\begin{equation}\label{ls}
   \dot y=f'({x}_0(t))y
\end{equation}
is equal to $1.$ In particular, our result implies that if $x_0$ is
a nondegenerate cycle of (\ref{np}) then the distance between the
sets $x_\epsi\left([0,T]\right)$ and $x_0\left([0,T]\right)$ is of
order $\epsi>0.$ Our result does not require differentiability of
$g,$ indeed here we assume that
\begin{equation}\label{reg0} \, f\in
C^1(\mathbb{R}^n,\mathbb{R}^n)\mbox{ \ \ and \ \ } g\in
C(\mathbb{R}\times\mathbb{R}^n\times [0,1],\mathbb{R}^n).
\end{equation}

\noindent This paper completes the existence and convergence
results of $T$-periodic solutions $x_\varepsilon$ of (\ref{ps})
proved in \cite{nach} under assumptions (\ref{reg0}). In fact, in
\cite{nach} we have observed in Remark~3.4 that the rate of
convergence of $x_\varepsilon([0,T])$ to ${x}_0([0,T])$ is of
order $\varepsilon^p$ with $0<p<1$. The convergence of
$x_\varepsilon([0,T])$ to ${x}_0([0,T])$ at rate $\epsi^1$ was
established in \cite{jmaa} for the case when $n=2,$ but instead of
$\Delta_\epsi$ we had $\Delta_\epsi(t)$ in (\ref{sqr1}). The
possibility of considering $\Delta_\epsi$ independent on time in
this paper is due to the considerable simplification of the
approach used in \cite{jmaa} that we have performed here.

\noindent The classical results on the existence and convergence
at rate $\epsi,$  of $T$-periodic solutions to equations of the
form  (\ref{ps}), where $\epsi>0$ is small, are due to Malkin
(\cite{mal}, Statement~p.~41) and Loud (\cite{loud}, Theorem~1)
where it is assumed that
\begin{equation}\label{reg1}
f\in C^2(\mathbb{R}^n,\mathbb{R}^n)\mbox{ \ \ and \ \ } g\in
C^1(\mathbb{R}\times\mathbb{R}^n\times [0,1],\mathbb{R}^n).
\end{equation}
Under less regularity assumptions the persistence of the limit cycle
$x_0$ is studied only for piecewise differentiable systems
(\ref{ps}), in fact in this case one can use the approach of A\u\i
zerman-Gantmaher \cite{ay}, Kolovski\u\i\ \cite{kol}, Lazer-McKenna
\cite{laz} and Ste\u\i nberg \cite{ste}. To our best knowledge
\cite{nach} and \cite{jmaa} are the first papers that provide
existence and convergence results of $T$-periodic solutions of
(\ref{ps}) bifurcating from a limit cycle $x_0$ under assumptions
(\ref{reg0}).

The paper is organized as follows. Section~2 is devoted to our main
result: Theorem~1 which states the validity of the inequality
(\ref{sqr1}). In Section~3 we apply (\ref{sqr1}) for studying some
further properties of convergence in (\ref{convp}), namely we
investigate
\begin{equation}\label{LIM}
  \lim_{\epsi\to 0}\frac{x_\epsi(t+\Delta_\epsi)-
  {x}_0(t)}{\epsi}
\end{equation}
by means of the first approximation. In particular, using only the
eigenfunctions $z$ of the adjoint system $\dot z=-(f'(x_0(t)))^*z$
and the function $g$ we give conditions ensuring that $
\lim_{\epsi\to
0}\left\|x_\varepsilon(\Delta_\varepsilon)-{x}_0(0)\right\|/\epsi=0
$ and we determine the signum of the angle between the vectors
$z(t)$ and $x_\varepsilon(t+\Delta_\varepsilon)- {x}_0(t).$  In the
smooth case (\ref{reg1}) these results may be derived, as discussed
in Remark~\ref{LOUD}, from the Loud's formula (\ref{lo}). Note that
formula (\ref{lo}) is established under the condition that a
suitably defined bifurcation function (\ref{biff}) has a simple
zero, our result does not require such a condition.

\vskip0.5truecm
\section{A formula for the distance
between the periodic solutions of the perturbed system and the
limit cycle of the unperturbed one}

\noindent In this Section we establish our main result, namely the
validity of (\ref{sqr1}). This result does not depend on the
perturbation term $g,$ indeed the only property we need is the
following.

\begin{definition} We say that the limit cycle ${x}_0$
of (\ref{np}) is nondegenerate
if the algebraic multiplicity\footnote{\label{fn}Let $Y(t)$ be the
normalized fundamental matrix of (\ref{ls}).  An eigenvalue $\rho$
of $Y(T)$ is called the characteristic multiplier of (\ref{ls}).
Since $\rho$ is a root of the corresponding characteristic
equation then it is possible to consider the multiplicity of this
root, which is called {\it algebraic multiplicity } of the
characteristic multiplier $\rho.$
System (\ref{ls}) always has at least one characteristic
multiplier $+1$ since from $\dot x_0(t)=f(x_0(t))$ we have $\ddot
x_0(t)=f'(x_0(t))\dot x_0(t).$} of the characteristic multiplier
$+1$ of (\ref{ls}) is equal to $1.$
\end{definition}

\noindent In order to introduce the family $\{\Delta_\epsi\}$ that
appears in (\ref{sqr1}) we define in what follows a suitable
surface $S\in C(\mathbb{R}^{n-1},\mathbb{R}^n).$ For this, let
$A_{n-1}$ be an arbitrary $n\times n-1$ matrix such that the
$n\times n$ matrix $(\dot x_0(0),A_{n-1})$ is nonsingular and
$\Omega(\cdot,t_0,\xi)$ is the solution of (\ref{np}) satisfying
$\Omega(t_0,t_0,\xi)=\xi.$ The surface $S$ is given by
\begin{equation}\label{suf}
\begin{array}{lll}
S(v)&=&\Omega(T,0,h(v)),\\
h(v)&=&{x}_0(0)+A_{n-1}v.
\end{array}
\end{equation}

\vskip0.3truecm The following result shows that the surface $S$
intersects $x_0$ transversally.

\begin{lemma}\label{conv_lem2}
Assume $f\in C^1(\mathbb{R}^n,\mathbb{R}^n).$ Let $x_0$ be a
nondegenerate $T$-periodic cycle of (\ref{np}). Then
$\dot{{x}}_0(0)\not\in S'(0)(\mathbb{R}^{n-1}).$
\end{lemma}

\begin{proof} We argue by contradiction, thus we assume that there exists
$v_0\in\mathbb{R}^{n-1},$ $v_0\not=0,$ such that
$\dot{{x}}_0(0)=S'(0)v_0.$ We have
\begin{equation}\label{T1}
\dot{x}_0(0)=S'(0)v_0=\Omega'_\xi(T,0,{x}_0(0))A_{n-1}v_0,
\end{equation}
where $\Omega'_{\xi}$ is the derivative of $\Omega$ with respect
to the third variable.

\noindent On the other hand, (see
\cite{kraop}, Theorem 2.1) $\Omega'_\xi(\cdot,0,x_0(0))$ is a
fundamental matrix to (\ref{ls}) and since $\dot x_0$ is a
solution to (\ref{ls}), we have that
\begin{equation}\label{T2}
\Omega'_\xi(T,0,x_0(0))\dot x_0(0)=\dot x_0(T)=\dot x_0(0).
\end{equation}
From (\ref{T1}) and (\ref{T2}) we conclude that $\dot
x_0(0)=A_{n-1}v_0,$ which means that the matrix $(\dot
x_0(0),A_{n-1})$ is singular contradicting the definition of
$A_{n-1}.$
\end{proof}

\vskip0.3truecm As a consequence of the previous lemma we have the
following result.

\begin{corollary}\label{conv_cor1} Assume the conditions of
Lemma~\ref{conv_lem2}. Let $x_\varepsilon$ be a $T$-periodic
solution to perturbed system (\ref{ps}) satisfying (\ref{convp})
uniformly with respect to $t\in\mathbb{R},$ then there exists
$\varepsilon_0>0$ and $r_0>0$ such that for any
$\varepsilon\in(0,\varepsilon_0]$ the equation
$x_\varepsilon(\Delta)=S(v)$ has an unique solution
$(\Delta_\varepsilon,v_\varepsilon)$ in
$[-r_0,r_0]\times\{v\in\mathbb{R}^{n-1}:\|v\|\le r_0\}.$ Moreover,
the functions $\varepsilon\to\Delta_\varepsilon,$  $\varepsilon\to
v_\varepsilon$ are continuous at $\varepsilon=0$ with the property
$\Delta_0=0$ and $v_0=0.$
\end{corollary}

\vskip0.3cm

\begin{proof} Define the function
$F\in C(\mathbb{R}^n\times[0,1],\mathbb{R}^n)$ as
$F((t,v),\varepsilon)=x_\varepsilon(t)-S(v),$ then $F((0,0),0)=0.$
Moreover, $F$ is continuously differentiable with respect to the
first variable and $F'_{(t,v)}((0,0),0)=(\dot x_0(0),-S'(0))$ is
nonsingular by Lemma~\ref{conv_lem2}. The conclusion follows from a
generalized version, see (\cite{kolm}, Ch.~X, \S~2.1), of the
classical implicit function theorem which requires $F\in
C^1(\mathbb{R}^n\times[0,1],\mathbb{R}^n),$ while we do not have
here the differentiability of $F$ with respect to the second
variable.
\end{proof}

\vskip0.3cm \noindent Figure~1 illustrates the meaning of
Corollary~\ref{conv_cor1}. We are now in the position to prove
inequality (\ref{sqr1}).

\begin{figure}
\begin{center}
  \includegraphics[scale=.6]{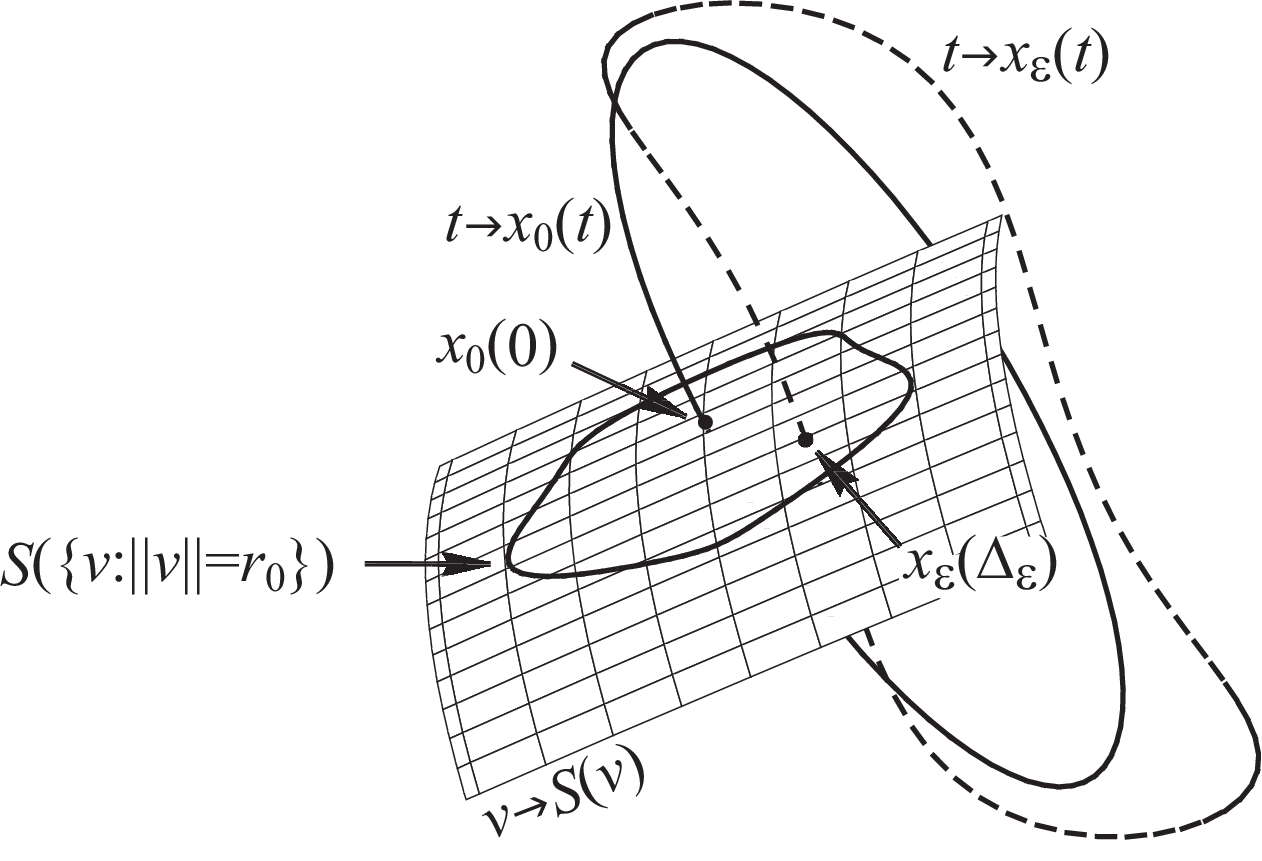}\\
  \caption{~}
  \end{center}
\end{figure}

\begin{theorem}\label{thmI} Assume $f\in C^1(\mathbb{R}^n,\mathbb{R}^n),$ $g\in C(\mathbb{R}\times
\mathbb{R}^n\times[0,1],\mathbb{R}^n).$ Let $x_\varepsilon$ be a
$T$-periodic solution to perturbed system (\ref{ps}) satisfying
\begin{equation}\label{conv}
  \|x_\varepsilon(t)- {x}_0(t)\|\to 0\mbox{\quad as}\quad\varepsilon\to
  0
\end{equation}
uniformly with respect to $t\in[0,T],$ where ${x}_0$ is a
nondegenerate $T$-periodic limit cycle of unperturbed system
(\ref{np}). Let $\epsi_0>0$ and
$\{\Delta_\epsi\}_{\epsi\in(0,\epsi_0]}\subset \mathbb{R}$ be as
in Corollary~\ref{conv_cor1}.
 Then there
exists $M>0$ such that
\begin{equation}\label{IN}
\|x_\varepsilon(t+\Delta_\varepsilon)- {x}_0(t)\|\le
M\epsi\quad\mbox{ for\ any\ }t\in[0,T] \mbox{ and\ any\ }\
\epsi\in(0,\epsi_0].
\end{equation}
\end{theorem}

\begin{proof} In the sequel $\epsi\in(0,\epsi_0]$ and
$\tau\in[0,T].$ Consider the change of variables
$\nu_\varepsilon(\tau)=\Omega(0,\tau,x_\varepsilon(\tau+\Delta_\varepsilon))$
in system (\ref{ps}). Observe that
\begin{equation}\label{PP}
x_\varepsilon(\tau+\Delta_\varepsilon)=\Omega(\tau,0,\nu_\varepsilon(\tau)).
\end{equation}
Taking the derivative in (\ref{PP}) with respect to $\tau$ we
obtain
\begin{equation}\label{ob1}
 \dot x_\varepsilon(\tau+\Delta_\varepsilon)=
 f(\Omega(\tau,0,\nu_\varepsilon(\tau))+\Omega'_{\xi}(\tau,0,
 \nu_\varepsilon(\tau)) \dot \nu_\varepsilon(\tau).
\end{equation}
On the other hand from (\ref{ps}) we have
\begin{equation}\label{ob2}
  \dot
  x_\varepsilon(\tau+\Delta_\varepsilon)=
  f(\Omega(\tau,0,\nu_\varepsilon(\tau)))+\varepsilon
  g(\tau+\Delta_\varepsilon,
  \Omega(\tau,0,\nu_\varepsilon(\tau)),\varepsilon).
\end{equation}
From (\ref{ob1}) and (\ref{ob2}) it follows
$$
  \dot \nu_\varepsilon(\tau)=\varepsilon\left(\Omega'_{\xi}
  (\tau,0,\nu_\varepsilon(\tau))\right)^{-1}g(\tau+\Delta_\varepsilon
  ,\Omega(\tau,0,\nu_\varepsilon(\tau)),\varepsilon),
$$
and since
$$
  \nu_\varepsilon(0)=x_\varepsilon(\Delta_\varepsilon)=x_\varepsilon(T+\Delta_\varepsilon)=
  \Omega(T,0,\nu_\varepsilon(T))
$$
we finally obtain
\begin{equation}\label{ob3}
  \nu_\varepsilon(\tau)=\Omega(T,0,\nu_\varepsilon(T))+\varepsilon\int\limits_0^\tau
  \left(\Omega'_{\xi}(s,0,\nu_\varepsilon(s))\right)^{-1}g(s+
  \Delta_\epsi,\Omega(s,0,\nu_\varepsilon(s)),\varepsilon)
  ds.
\end{equation}
Since $\nu_\varepsilon(\tau)\to {x}_0(0)$, for any $\tau\ge 0,$ as
$\varepsilon\to 0$ we can write $\nu_\varepsilon(\tau)$ in the
following form
\begin{equation}\label{rep}
  \nu_\varepsilon(\tau)={x}_0(0)+\varepsilon
  \mu_\varepsilon(\tau).
\end{equation}
We now prove that the functions $\mu_\varepsilon$ are bounded on
$[0,T]$ uniformly with respect to $\varepsilon\in(0,\varepsilon_0].$
For this, we first subtract ${x}_0(0)$ from both sides of
(\ref{ob3}), with $\tau=T,$ obtaining
\begin{eqnarray}
\varepsilon\mu_\varepsilon(T)&=&\varepsilon\,\Omega'_{\xi}(T,0,{x}_0(0))
\mu_\varepsilon(T)+o(\varepsilon\mu_\varepsilon(T))\nonumber\\
& &
+\varepsilon\,\int_0^T\left(\Omega'_{\xi}(s,0,\nu_\varepsilon(s))
\right)^{-1}g(s+\Delta_\varepsilon,
\Omega(s,0,\nu_\varepsilon(s)),\varepsilon)
  ds,\label{ob4}
\end{eqnarray}
where, from (\ref{rep}),
$\dfrac{o(\epsi\mu_\epsi(T))}{\|\epsi\mu_\epsi(T)\|}\to 0$ as
$\epsi\to 0.$ Since $x_\varepsilon(\Delta_\varepsilon)\in
S\left(\{v\in\mathbb{R}^{n-1}:\|v\|\le r_0\right)$ then by
Corollary~\ref{conv_cor1} there exists $v_\varepsilon\in
\mathbb{R}^{n-1},$ $\|v_\epsi\|\le r_0,$ such that
\begin{equation}\label{SP}
x_\varepsilon(\Delta_\varepsilon)=
\Omega(T,0,x_0(0)+A_{n-1}v_\epsi)
\end{equation}
and
\begin{equation}\label{convr}
v_\varepsilon\to 0\quad{\rm as}\quad \varepsilon\to 0.
\end{equation}
Now by using (\ref{SP}) we can represent
$\varepsilon\mu_\varepsilon(T)$ as follows
\begin{equation}\label{key}
\begin{array}{lll}
\varepsilon\mu_\varepsilon(T) & = &
\nu_\varepsilon(T)-{x}_0(0)=\Omega(0,T,x_\varepsilon
(\Delta_\varepsilon))-{x}_0(0)\\
&= &
\Omega(0,T,\Omega(T,0,x_0(0)+A_{n-1}v_\epsi))-{x}_0(0)=A_{n-1}v_\varepsilon.
\end{array}
\end{equation}
Therefore  (\ref{ob4}) can be rewritten as follows
\begin{eqnarray}
A_{n-1}v_\varepsilon&=&\Omega'_{\xi}(T,0,{x}_0(0))A_{n-1}v_\varepsilon+
o(A_{n-1}v_\varepsilon)\nonumber\\
& &+\,
\varepsilon\int_0^T\left(\Omega'_{\xi}(s,0,\nu_\varepsilon(s))
\right)^{-1}g(s+\Delta_\varepsilon,
\Omega(s,0,\nu_\varepsilon(s)),\varepsilon)
  ds.\label{bisp}
\end{eqnarray}
Let us show that there exists $M_1>0$ such that
\begin{equation}\label{estt}
  \|v_\varepsilon\|\le \varepsilon
  M_1,\qquad\mbox{for any}\quad \varepsilon\in(0,\varepsilon_0].
\end{equation}
Arguing by contradiction we assume that there exist sequences
$\{\varepsilon_k\}_{k\in\mathbb{N}}\subset(0,\epsi_0],$
$\varepsilon_k\to 0$ as $k\to\infty,$ such that
$\|v_{\varepsilon_k}\|=\varepsilon_k c_k,$ where $c_k\to\infty$ as
$k\to\infty.$ Let
$q_k=\dfrac{v_{\varepsilon_k}}{\|v_{\varepsilon_k}\|},$ then from
(\ref{bisp}) we have
\begin{eqnarray}\label{ff}
 & & A_{n-1}q_k=\Omega'_{\xi}(T,0,{x}_0(0))A_{n-1} q_k+\frac{o(A_{n-1}v_{\varepsilon_k})}
 {\|v_{\varepsilon_k}\|}\nonumber\\
 &  &+\frac{1}{c_k}\,\int_0^{T}
 \left(\Omega'_{\xi}(s,0,\nu_{\varepsilon_k}(s))
 \right)^{-1}g(s+\Delta_{\varepsilon_k},\Omega(s,0,\nu_{\varepsilon_k}(s)),\varepsilon_k)
  ds,
\end{eqnarray}
where $\dfrac{o(A_{n-1}v_{\varepsilon_k})}{\|v_{\varepsilon_k}\|}\to
0$ as $k\to \infty,$ in fact
$\dfrac{o(A_{n-1}v_{\varepsilon_k})}{\|v_{\varepsilon_k}\|}=\dfrac{o(A_{n-1}v_{\varepsilon_k})}{\|A_{n-1}v_{\varepsilon_k}\|}\cdot
\dfrac{\|A_{n-1}v_{\varepsilon_k}\|}{v_{\varepsilon_k}}.$
 Without
loss of generality we may assume that the sequence
$\{q_k\}_{k\in\mathbb{N}}$ converges, let
$q_0=\lim_{k\to\infty}q_k$ with $\|q_0\|=1$. By passing to the
limit as $k\to\infty$ in (\ref{ff}) we have that
$$
A_{n-1}q_0=\Omega'_{\xi}(T,0,{x}_0(0))A_{n-1} q_0.
$$
Therefore $A_{n-1}q_0$ is the initial condition of a $T$-periodic
solutions to (\ref{ls}). On the other hand the cycle $x_0$ is
nondegenerate, hence $A_{n-1}q_0$ is linearly dependent with $\dot
x_0$ contradicting the choice of $A_{n-1}.$ Thus (\ref{estt}) is
true for some $M_1>0.$ From (\ref{rep}) and the fact that
$\nu_\varepsilon(0)=x_\varepsilon(\Delta_\varepsilon)$ we have
\begin{equation}\label{DEC}
\begin{array}{lll}
 \|x_\epsi(\Delta_\epsi)-{x}_0(0)\|&=&\varepsilon\|\mu_\varepsilon(0)\|\le
 \epsi\|\mu_\epsi(T)\|+\|\epsi\mu_\epsi(T)-\epsi\mu_\epsi(0)\|\\
 &=&\epsi\|\mu_\epsi(T)\|+\|\nu_\epsi(T)-\nu_\epsi(0)\|.
\end{array}
\end{equation}
From (\ref{ob3}) we have that there exists $M_2>0$ such that
\begin{equation}\label{estt1}
  \|\nu_\epsi(T)-\nu_\epsi(0)\|\le \epsi M_2,\qquad{\rm for\ any\
  }\epsi\in(0,\epsi_0].
\end{equation}
Therefore combining (\ref{key}) with (\ref{estt})  and taking into
account (\ref{estt1}) we have from (\ref{DEC}) that
$$
  \|x_\epsi(\Delta_\epsi)-{x}_0(0)\|\le \epsi\|A_{n-1}\|
  M_1  +\epsi M_2,\qquad{\rm for\ any\
  }\epsi\in(0,\epsi_0].
$$
Since
$$
\begin{array}{ll}
x_\epsi(t+\Delta_\epsi)-x_0(t)=&x_\epsi(\Delta_\epsi)-x_0(0)+\int\limits_0^t(f(x_\epsi(s+\Delta_\epsi))-
f(x_0(s)))ds\\
&+\epsi\int\limits_0^t
g(s+\Delta_\epsi,x_\epsi(s+\Delta_\epsi),\epsi)ds
\end{array}
$$
and $f\in C^1(\mathbb{R}^n,\mathbb{R}^n)$ then there exist a
constant $M_3\ge 0$ such that
\begin{equation}\label{GR}
\begin{array}{lll}
\|x_\epsi(t+\Delta_\epsi)-x_0(t)\|&\le&  \left(\epsi\|A_{n-1}\|M_1+
 \epsi M_2\right)\\
 & &+M_3\int\limits_0^t\|x_\epsi(s+\Delta_\epsi)-
x_0(s)\|ds+\epsi M_3,
\end{array}
\end{equation}
for any $\epsi\in(0,\epsi_0].$ By means of the Gronwall-Bellman
lemma (see e.g. \cite{dem}, Ch.~II, \S~11) inequality (\ref{GR})
implies
\begin{equation*}
\|x_\epsi(t+\Delta_\epsi)-x_0(t)\|\le  \epsi\left(\|A_{n-1}\|
  M_1+ M_2+M_3\right){\rm e}^{M_3 T}\quad{\rm for\ any\
  }\epsi\in(0,\epsi_0].
\end{equation*}
and thus the proof is complete.
\end{proof}

\begin{remark}\label{new} Assume that the $T$-periodic solution
$x_\epsi$ of system (\ref{ps}) satisfies the property
$\|x_\epsi(t)-\widetilde{x}(t)\|\to 0,$ where  $\widetilde{x}$ is a
$T$-periodic nondegenerate limit cycle of (\ref{np}). Let
$\tau\in[0,T],$ define $x_0^\tau(t):=\widetilde{x}(t+\tau),$ then we
have $\|x_\epsi(t+\tau)-x_0^\tau(t)\|\to 0$ as $\epsi\to 0.$ Denote
by $S^\tau$ the surface $S$ corresponding to $x_0^\tau$ given by
(\ref{suf}). Observe now that $\epsi_0>0$ of
Corollary~\ref{conv_cor1} can be chosen sufficiently small in such a
way that it does not depend on the choice of $\tau\in[0,T]$ used in
the definition of $x_0^\tau.$ Therefore for any $\tau\in[0,T]$ and
any $\epsi\in(0,\epsi_0]$ Corollary~\ref{conv_cor1} guarantees the
existence of a $\Delta_{\epsi}^\tau$ with the property
\begin{equation}\label{R1}
  S^\tau(\{v\in\mathbb{R}^{n-1}:\|v\|\le r_0\})\cap
  x_\epsi([0,T])=\{x_\epsi(\Delta_{\epsi}^\tau+\tau)\}.
\end{equation}
In conclusion, Theorem~\ref{thmI} can be proved by replacing in
(\ref{IN}) $\Delta_\epsi$ by $\Delta_{\epsi}^\tau,$ namely one has
the following conclusion
\begin{equation}\label{R2}
  \|x_\epsi(t+\tau+\Delta_{\epsi}^\tau)-x_0^\tau(t)\|\le
  M\epsi\quad{\rm for\ any\ }t,\tau\in[0,T]{\rm\ and\ any\ }\epsi\in
  (0,\epsi_0].
\end{equation}
In particular, if for any $\epsi>0$ sufficiently small there exists
$\tau_\epsi\in[0,T]$ such that $x_\epsi(\tau_\epsi)$ belongs to
$S^{\tau_\epsi}$ then we can take $\Delta_{\epsi}^{\tau_\epsi}=0$ in
(\ref{R1}) and (\ref{R2}) becomes
$$
\|x_\varepsilon(t)- \widetilde{x}(t)\|\le M\epsi\quad\mbox{for
any}\ t\in[0,T]\ \mbox{and}\ \epsi>0\ \mbox{sufficiently  small}.
$$
\end{remark}

\begin{remark} It can be checked, see e.g. (\cite{can},
formula~37), that the cycle $ {x}_0 $ of the example in the
introduction is nondegenerate. Thus Theorem~\ref{thmI} applies to
the example to ensure the existence of $\Delta_\epsilon = (\sqrt
\varepsilon)$ for which the left hand side of (\ref{sqr}) is
uniformly bounded with respect to $t\in [0, 2\pi]$ and
$\varepsilon>0$ sufficiently small.
\end{remark}

\section{Applications}

\vskip0.3truecm \noindent Let $z$ be an eigenfunction of the adjoint
system of (\ref{ls})
\begin{equation}\label{ss}
  \dot z=-(f'({x}_0(t)))^*z,
\end{equation}
which is not $T$-periodic. Here and in the following $^*$ denotes
the transpose. We recall that an eigenfunction of a linear
$T$-periodic system is a Floquet solution of this system, namely it
is a solution $z$ satisfying $z(t+T)=\rho z(t)$ for some
$\rho\in\mathbb{R}$ and any $t\in\mathbb{R}.$ Consider the scalar
product
\begin{equation}\label{SPr}
\left<z(t),\lim_{\epsi\to 0}\frac{x_\epsi(t+\Delta_\epsi)-
  {x}_0(t)}{\epsi}\right>.
\end{equation}
In this Section we provide several results about the convergence of
$x_\varepsilon(t+\Delta_\varepsilon)$ to ${x}_0(t)$ in terms of
(\ref{SPr}). The main tool is the following scalar function
\begin{equation}\label{tool}
  M^\bot_z(t)= \frac{\rho}{\rho-1}\int\limits_{t-T}^t
  \left<z(s),g(s,{x}_0(s),0)\right>ds,
\end{equation}
where $\rho$ is the characteristic multiplier of (\ref{ss})
corresponding to the eigenfunction $z.$ The relationship between
(\ref{SPr}) and $M^\bot_z$ is shown by the following result.

\begin{theorem}\label{conv_thm1} Assume $f\in C^1(\mathbb{R}^n,\mathbb{R}^n),$ $g\in C(\mathbb{R}\times
\mathbb{R}^n\times[0,1],\mathbb{R}^n).$ Let $x_\varepsilon$ be a
$T$-periodic solution to (\ref{ps}) such that
\begin{equation}\label{INp}
\|x_\varepsilon(t+\Delta_\varepsilon)- {x}_0(t)\|\le
M\epsi\quad\mbox{for\ any\ }t\in[0,T]\ \mbox{and any}\
\epsi\in(0,\epsi_0],
\end{equation}
where $\Delta_\epsi\to 0$ as $\epsi\to 0,$ $M, \epsi_0>0$ and $x_0$
is a nondegenerate limit cycle of (\ref{np}). Let $z$ be a not
$T$-periodic eigenfunction of (\ref{ss}). Then
\begin{equation}\label{mainprop}
  \lim_{\epsi\to
0}\dfrac{1}{\epsi}\left<z(t),x_\varepsilon(t+\Delta_\varepsilon)-
{x}_0(t)\right> =M^\bot_z(t)
\end{equation}
uniformly with respect to $t\in[0,T].$
\end{theorem}


\begin{proof} In the sequel
$\epsi\in(0,\epsi_0],$ and $t,\tau\in[0,T].$ Let $A$ be a
nonsingular $n\times n$ matrix such that
\begin{equation}\label{Z0}
  z(0)^*\,A=(0,...,0,1).
\end{equation}
 Let $Y(t)$ be the fundamental matrix of the linearized system
(\ref{ls}) with initial condition $Y(0)=A.$ Since $A$ is nonsingular
then the columns of $Y(t)$ are linearly independent. Let
\begin{equation}\label{FR}
   Z(t)=\left(Y(t)^*\right)^{-1}
\end{equation}
and $a_\epsi\in C([0,T],\mathbb{R}^n)$ is given by
$$
  a_{\epsi}(t)=Z(t)^*\,\frac{x_\epsi(t+\Delta_\epsi)-{x}_0(t)}{\epsi}.
$$
Then we have
\begin{equation}\label{fw}
x_\epsi(t+\Delta_\epsi)-{x}_0(t)=\epsi Y(t)a_\epsi(t),
\end{equation}

In what follows by $o(\varepsilon),$ $\varepsilon>0,$ we will
denote a function, which may depend also on other variables,
having the property that $\dfrac{o(\varepsilon)}{\varepsilon}\to
0$ as $\varepsilon\to 0$ uniformly with respect to these variables
when they belong to any bounded set.

 By subtracting (\ref{np}) where $x(t)$ is replaced by
${x}_0(t)$ from (\ref{ps}) where $x(t)$ is replaced by
$x_\varepsilon(t+\Delta_\varepsilon)$ we obtain
\begin{eqnarray}\label{bis1}
  \dot x_\varepsilon(t+\Delta_\varepsilon)-\dot{
  {x}}_0(t)&=&f'({x}_0(t))(x_\varepsilon(t+\Delta_\varepsilon)-{x}_0(t))\nonumber\\
   & & +\varepsilon g(t+\Delta_\varepsilon,x_\varepsilon(t+\Delta_\varepsilon),\varepsilon)+o_t(\varepsilon),
\end{eqnarray}
here $\varepsilon\to  o_t(\varepsilon)$ is such that $
o_{t+T}(\cdot)=o_t(\cdot)$ for any $t\in\mathbb{R}.$ By
substituting (\ref{fw}) into (\ref{bis1}) we have
$$
  \epsi \dot Y(t)a_\epsi(t)+\epsi Y(t)\dot a_\epsi(t)
$$
$$
 =\epsi
 f'({x}_0(t))Y(t)a_\epsi(t)+\epsi g(t+\Delta_\varepsilon,x_\varepsilon(t+\Delta_\varepsilon),\varepsilon)+
  o_t(\varepsilon).
$$
Since $f'({x}_0(t))Y(t)=\dot Y(t)$ the last relation can be
rewritten as
\begin{equation}\label{BS}
  \epsi Y(t)\dot a_\epsi(t)
 =\epsi g(t+\Delta_\varepsilon,x_\varepsilon(t+\Delta_\varepsilon),\varepsilon)+
  o_t(\varepsilon).
\end{equation}
By means of Perron's lemma \cite{perron} (see also Demidovich
(\cite{dem}, Sec. III, \S 12) formula (\ref{Z0}) implies that
\begin{equation}\label{IM}
  z(t)^*\, Y(t)=(0,...,0,1)\quad{\rm for\ any\ }t\in\mathbb{R}.
\end{equation}
Therefore, applying $z(t)^*$ to both sides of (\ref{BS}) we have
$$
  \epsi(a_{\epsi,n})'(t)=\epsi z(t)^*\, g(t+\Delta_\varepsilon,x_\varepsilon(t+\Delta_\varepsilon),\varepsilon)
  +z(t)^*\, o_t(\varepsilon),
$$
where $a_{\epsi,n}(t)$ is the $n$-th component of the vector
$a_\epsi(t),$ and so
\begin{equation}\label{prob}
\begin{array}{rcl}
  a_{\varepsilon,n}(t)&=&a_{\varepsilon,n}(t_0)+\int\limits_{t_0}^t
  \left<z(\tau),g(\tau+\Delta_\varepsilon,
  x_\varepsilon(\tau+\Delta_\varepsilon),\varepsilon)\right>d\tau\\
  & & +\int\limits_{t_0}^t\left<z(\tau),\dfrac{o_\tau(\epsi)}{\epsi}\right>d\tau.
\end{array}
\end{equation}
From (\ref{FR}) we have that $Z(0)^*\,Y(0)=I.$ Therefore
$$
  \left(\left[Z(0)\right]_n\right)^* A=(0,...,0,1),
$$
where $[Z(0)]_n$ is the $n$-th column of $Z(0).$ Thus
$[Z(0)]_n=z(0)$ and so $a_{\epsi,n}(t)$ satisfies
\begin{equation}\label{bc}
  a_{\varepsilon,n}(t_0)=\rho a_{\varepsilon,n}(t_0-T)\quad{\rm
  for\ any\ }t_0\in[0,T].
\end{equation}
Solving (\ref{prob})-(\ref{bc}) with respect to
$a_{\varepsilon,n}(t_0)$ we obtain
$$
\begin{array}{rcl}
  a_{\varepsilon,n}(t_0)
&=&\dfrac{\rho}{\rho-1}\int\limits_{t_0-T}^{t_0}
\left<z(\tau),g(\tau+\Delta_\epsi,
  x_\varepsilon(\tau+\Delta_\epsi),\epsi)\right>d\tau\\
  & & +\dfrac{\rho}{\rho-1}\int\limits_{t_0-T}^{t_0}
  \left<z(\tau),\dfrac{o_\tau(\varepsilon)}{\epsi}\right>d\tau \quad{\rm for\ any\
  }t_0\in[0,T].
\end{array}
$$
On the other hand taking the scalar product of (\ref{fw}) with
$z(t)$ and using (\ref{IM}) we
 obtain
$$
\left<z(t),x_\varepsilon(t+\Delta_\varepsilon)-{x}_0(t)\right>=\varepsilon
a_{\varepsilon,n}(t)
$$
and thus
$$
 \begin{array}{rl}
    \dfrac{1}{\epsi}\left<z(t), x_\epsi(t+\right.&\hskip-
    0.3cm\left.\Delta_\epsi)-x_0(t)\right>=\dfrac{\rho}{\rho-1}
    \int\limits_{t_0-T}^{t_0}\left<z(\tau),g(\tau,x_0(\tau),0),0)\right>d\tau\\
    &
    +\dfrac{\rho}{\rho-1}\int\limits_{t_0-T}^{t_0}\left<z(\tau),
    g(\tau+\Delta_\epsi,x_\epsi(\tau+\Delta_\epsi),\epsi)-
    g(\tau,x_0(\tau),0)\right>d\tau\\
    &+\dfrac{\rho}{\rho-1}\int\limits_{t_0-T}^{t_0}\left<z(\tau),
    \dfrac{o_\tau(\epsi)}{\epsi}\right>d\tau.
 \end{array}
$$
The proof is complete.
\end{proof}

\begin{remark} \label{bif} It is of some interest to compare the
scalar function $M_z^\bot$ introduced in (\ref{tool}) with the
Malkin's bifurcation function (see \cite{mal}, formula~3.13) that
for system (\ref{ps}) takes the form
\begin{equation}\label{biff}
  M_{z_0}(t)= \int\limits_0^T
  \left<z_0(s),g(s-t,{x}_0(s),0)\right>ds,
\end{equation}
where $z_0$ is a $T$-periodic solution of system (\ref{ss}). This
bifurcation function was employed by Loud (see \cite{loud}
formula~3.48) to study for system (\ref{ps}) the convergence of
$x_\epsi$ to $x_0$.
\end{remark}

\begin{remark} \label{LOUD} Under the regularity assumptions
(\ref{reg1}), Malkin in \cite{mal} and Loud in \cite{loud} proved
that if
\begin{equation}\label{ML}
  M_{z_0}(0)=0{\rm \ and\ }(M_{z_0})'(0)\not=0,
\end{equation}
then (\ref{IN}) is valid with $\Delta_\epsi=0.$ Furthermore, letting
\begin{equation}\label{INVE}
  y_0(t)=\lim_{\epsi\to 0}\dfrac{x_\epsi(t)-x_0(t)}{\epsi}
\end{equation}
Malkin (\cite{mal}, formulas~4.3-4.4) showed that $y_0$ is a
$T$-periodic solution of
$$
  \dot y=f'(x_0(t))y+g(t,x_0(t),0)
$$
and Loud (\cite{loud}, formula~1.3, Lemma~1 and formula for $x$ at
p.~510) up to a change of coordinate represented $y_0$ in the form
\begin{equation}\label{lo}
\begin{array}{ll}
 y_0(t)=&\Phi_0(t)\int\limits_0^t\Phi_0^{-1}(s)g(s,x_0(s),0)ds+C\dot
 x_0(t)\\
 &+\,\Phi_0(t)\left\|\begin{array}{rr}
  0 & \alpha^{-1}\beta(B-I)^{-1}B\\
  0 &
  -(B-I)^{-1}B\end{array}\right\|\int\limits_0^T\Phi_0^{-1}(s)g(s,x_0(s),0)ds,
\end{array}
\end{equation}
where $\Phi_0$ is a suitable fundamental matrix of the linearized
system  (\ref{ls}), $\alpha\in\mathbb{R},$
$\beta^*\in\mathbb{R}^{n-1},$ and $B$ is a $n-1\times n-1$ matrix
defined by means of $\Phi_0(0)$ and $\Phi_0(T).$ Assuming that $z$
is not a $T$-periodic eigenfunction of (\ref{ss}) the question if
(\ref{INVE})-(\ref{lo}) imply (\ref{mainprop}) (with
$\Delta_\epsi=0$) was not addressed in the previous papers. This
means that our Theorem~\ref{conv_thm1} represents also a
contribution to the case when $f,g$ satisfy (\ref{reg1}) without
assuming (\ref{ML}).
\end{remark}

\vskip0.5truecm \noindent In the sequel by using (\ref{mainprop})
of Theorem~\ref{conv_thm1} and the properties of $M^\bot_z$
several results about the behavior  of
$\dfrac{x_\epsi(t+\Delta_\epsi)-x_0(t)}{\epsi}$ as $\epsi\to 0$
are given. By
$\cos\angle(a,b)=\dfrac{\left<a,b\right>}{\|a\|\cdot\|b\|}$ we
denote the cosine of the  angle between the vectors $a,\,
b\in\mathbb{R}^n.$

\vskip0.5truecm\noindent
\begin{corollary}\label{conv_cor3} Assume all the conditions of
Theorem~\ref{conv_thm1}. Then for any eigenfunction $z$ of
(\ref{ss}), any $t\in[0,T]$ and sufficiently small $\varepsilon>0$
we have
$$ \mbox {if}\quad M^\bot_z(t)>0\quad\mbox{then}\quad
\cos\angle\left(z(t), x_\varepsilon(t+\Delta_\varepsilon)-
  {x}_0(t)
  \right)>0,
$$
$$ \mbox {if}\quad M^\bot_z(t)<0\quad\mbox{then}\quad
\cos\angle\left(z(t), x_\varepsilon(t+\Delta_\varepsilon)-
  {x}_0(t)
  \right)<0.
$$
\end{corollary}

\noindent To prove Corollary~\ref{conv_cor3} it is sufficient to
observe that
\begin{equation}\label{mainprop1}
  \lim_{\epsi\to 0}\frac{1}{\epsi}\left\|z(t)\right\|\, \left\|x_\varepsilon(t+\Delta_\varepsilon)-
  {x}_0(t)\right\|\, \cos\angle\left(z(t),
x_\varepsilon(t+\Delta_\varepsilon)-
  {x}_0(t)
  \right)
= M^\bot_z(t)
\end{equation}
obtained by substituting
$$
\begin{array}{l}
  \left<z(t),x_\varepsilon(t+\Delta_\varepsilon)-{x}_0(t)\right>\\
  \hskip1cm=\left\|z(t)\right\|\, \left\|x_\varepsilon(t+\Delta_\varepsilon)-
  {x}_0(t)\right\|\, \cos\angle\left(z(t),
x_\varepsilon(t+\Delta_\varepsilon)-
  {x}_0(t)
  \right)
\end{array}
$$
into formula (\ref{mainprop}).

\vskip0.3truecm The next result is a direct consequence of
(\ref{mainprop1}).

\begin{corollary}\label{conv_cor4} Assume all the conditions of
Theorem~\ref{conv_thm1}. If there exists a not $T$-periodic
eigenfunction $z$ to (\ref{ss}) such that  $M^\bot_z(t)\not=0$ for
any $t\in[0,T]$ then
$$
  c_1\varepsilon\le\left\|x_\varepsilon(t+\Delta_\varepsilon)-
  {x}_0(t)\right\|\le
  c_2\varepsilon\quad{\rm for\ any\ }t\in[0,T]
$$
for some $0<c_1\le c_2,$ and sufficiently small $\varepsilon>0.$
\end{corollary}

Combining Theorem~\ref{thmI} and Theorem~\ref{conv_thm1}  we can
derive the following fact.
\begin{corollary}\label{conv_cor5} Assume all the conditions of
Theorem~\ref{thmI}. Assume that $T>0$ is the least period of
$x_0.$ If there exists a not $T$-periodic eigenfunction $z$ to
(\ref{ss}) such that $M^\bot_z(0)\not=0$ then
$$
x_\varepsilon(t)\not={x}_0(0)\quad \mbox{ for any }\; t\in[0,T]
$$
provided that $\varepsilon>0$ is sufficiently small.
\end{corollary}

\begin{proof} Equivalently we prove that
$x_\varepsilon(t)\not=x_0(0)$ for any $t\in[-T/2,T/2].$ Arguing by
contradiction we assume that there exist sequences $\varepsilon_k\to
0$ and $[-T/2,T/2]\ni t_k\to t_0$ as $k\to\infty$ such that
\begin{equation}\label{L1}
   x_{\varepsilon_k}(t_k)=x_0(0)\quad{\rm for\ any\
   }k\in\mathbb{N}.
\end{equation}
We have that
\begin{equation}\label{PRO}
  t_k=\Delta_{\varepsilon_k}{\rm\ for\ }k\in\mathbb{N}{\rm\
  sufficiently\ large,}
\end{equation}
where $\Delta_{\varepsilon_k}$ are given by
Corollary~\ref{conv_cor1}. Indeed, if (\ref{PRO}) does not hold
then from Corollary~\ref{conv_cor1} we obtain that $0<r_0<|t_k|\le
T/2.$ Therefore, passing to the limit in (\ref{L1}) we have
$x_0(t_0)=x_0(0)$ with $0<|t_0|\le T/2.$ This contradicts the fact
that $T>0$ is the least period of $x_0$ and so (\ref{PRO}) holds
true. Hence, from (\ref{L1}) we have that
$x_{\varepsilon_k}(\Delta_{\varepsilon_k})=x_0(0)$ for any
$k\in\mathbb{N}$ sufficiently large and passing to the limit as
$k\to\infty$ in (\ref{mainprop}) with $\varepsilon=\varepsilon_k$
we obtain $M^\bot_z(0)=0$ contradicting our assumptions.
\end{proof}

Observe that Corollary~\ref{conv_cor5}, unlike the other
Corollaries of this Section, requires that the
$\Delta_\varepsilon$ in (\ref{INp}) are those given by
Corollary~\ref{conv_cor1}. This is the reason for assuming the
conditions of Theorem~\ref{thmI} in Corollary~\ref{conv_cor5}.

\vskip0.5truecm\noindent In order to establish sufficient conditions
to ensure that the convergence rate of $x_\varepsilon$ to ${x}_0$ as
$\varepsilon\to 0$ is of order greater than $\varepsilon^1$ we need
the following preliminary result.

\begin{lemma}\label{conv_lem3} Assume all the conditions of
Theorem~\ref{conv_thm1}. Assume that system (\ref{ss}) has $n$
linearly independent eigenfunctions. If $\varepsilon_0>0$ is
sufficiently small, then for every
$\varepsilon\in(0,\varepsilon_0]$ such that
$$x_\varepsilon(\Delta_\varepsilon)\not={x}_0(0)$$
 there exists a not $T$-periodic eigenfunction $z$ of (\ref{ss})
 satisfying
\begin{equation}\label{lemma3}
\left|\cos\angle\left(z(0),x_\varepsilon(\Delta_\varepsilon)-{x}_0(0)\right)
\right|\ge\alpha_*,
\end{equation}
where $\alpha_*>0$ does not depend on $\varepsilon.$
\end{lemma}

To prove Lemma~\ref{conv_lem3} we need the following property of
eigenfunctions of (\ref{ss}).

\begin{lemma}\label{EI} Assume that $f\in C^1(\mathbb{R}^n,\mathbb{R}^n)$ and let $x_0$ be
a nondegenerate $T$-periodic limit cycle of (\ref{np}). Assume
that system (\ref{ss}) has $n$ linearly independent
eigenfunctions. Denote by $z_1,...,z_{n-1}$ the eigenfunctions of
(\ref{ss}) which are not $T$-periodic. Let ${z_0}$ be the
$T$-periodic eigenfunction of (\ref{ss}) such that
\begin{equation}\label{orient}
\left<\dot{{x}}_0(0),{z_0}(0)\right>=1.
\end{equation}
Then the last column of
$\left((z_1(t),...,z_{n-1}(t),{z_0}(t))^{-1}\right)^*$ is
$\,\dot{{x}}_0(t).$
\end{lemma}

\begin{proof} Let $\left(y_1(t),...,y_n(t)\right)=
\left((z_1(t),...,z_{n-1}(t),{z_0}(t))^{-1}\right)^*, \,t\in
\mathbb{R}.$ We want to show that $y_n(t)=\dot{{x}}_0(t), \,t\in
\mathbb{R}.$ Since
$$
\left(y_1(t),...,y_n(t)\right)^*
\left(z_1(t),...,z_{n-1}(t),{z_0}(t)\right)=I
$$
then \begin{equation}\label{OL} \left<y_n(t),z_i(t)\right>=0\quad
{\rm for\ any\ } i\in \overline{1,n-1}{\rm\ and\ }
\left<y_n(t),{z_0}(t)\right>=1, \end{equation} whenever $t\in
\mathbb{R}$. Let us show that
\begin{equation}\label{orient0}
  \left<\dot{{x}}_0(0),z_i(0)\right>=0,\qquad \mbox{for
  any}\quad
  i\in\overline{1,n-1}.
\end{equation}
Indeed, let $i\in\overline{1,n-1}.$ Since eigenfunctions $z_i$ are
not $T$-periodic then $\rho_i z_i(0)=z_i(T)$ for some $\rho_i\not=1$
that gives
$\rho_i\left<\dot{{x}}_0(0),z_i(0)\right>=\left<\dot{{x}}_0(T),z_i(T)\right>.$
On the other hand Perron's lemma \cite{perron} implies that
$\left<\dot{{x}}_0(0),z_i(0)\right>=\left<\dot{{x}}_0(T),z_i(T)\right>,$
thus $\left<\dot{{x}}_0(0),z_i(0)\right>=0$ and so (\ref{orient0})
holds true. But the vectors $z_1(0), z_2(0),...., {z_0}(0)$ form a
basis of $\mathbb{R}^n$, hence by (\ref{orient}), (\ref{OL}) and
(\ref{orient0}) we get $y_n(0)=\dot{ {x}}_0(0)$.
\end{proof}

\begin{proof}[Proof of Lemma~\ref{conv_lem3}.] Let
$z_1,...,z_{n-1}$ be linearly independent not $T$-periodic
eigenfunctions of (\ref{ss}). Let us show that there exists
$i_*\in\overline{1,n-1}$ such that the conclusion holds with $z$
replaced by $z_{i_*}.$ Assume the contrary, therefore there exist
$\{\varepsilon_k\}_{k\in\mathbb{N}}\subset(0,1],$
$\{\alpha_k\}_{k\in\mathbb{N}}\subset[-1,1]$ with
$\varepsilon_k\to 0$ and $\alpha_k\to 0$ as $k\to\infty$ such that
$$
\cos\angle\left(z_{i}(0),x_{\epsi_k}(\Delta_{\epsi_k})-x_0(0)\right)=\alpha_k\quad\mbox{for
any }i\in\overline{1,n-1}
$$
and
$$
  x_{\epsi_k}(\Delta_{\epsi_k})\not=x_0(0).
$$
Let $v_k\in\mathbb{R}^{n-1}$ be such that
$S(v_k)=x_{\epsi_k}(\Delta_{\epsi_k}).$ Therefore we have
$$
\frac{\left<z_i(0),S(v_k)-S(0)\right>}
{\left\|z_i(0)\right\|\cdot\left\|S(v_k)-S(0)\right\|}=\alpha_k\quad\mbox{for
any }i\in\overline{1,n-1},
$$
and so
\begin{equation}\label{ptl}
\frac{\left<z_i(0),\dfrac{S(v_k)-S(0)}{\|v_k\|}\right>}
{\left\|z_i(0)\right\|\cdot\left\|\dfrac{S(v_k)-S(0)}{\|v_k\|}\right\|}=\alpha_k\quad\mbox{for
any }i\in\overline{1,n-1}.
\end{equation}
Passing to a subsequence if necessary, we may assume that
$\dfrac{v_k}{\|v_k\|}\to q_0$ as $k\to\infty,$ thus $\|q_0\|=1.$
Taking the limit as $k\to\infty$ in (\ref{ptl}) we obtain
\begin{equation}\label{zvio}
  \left<z_i(0),S'(0)q_0\right>=0\quad\mbox{for any
}i\in\overline{1,n-1}.
\end{equation}
Since $q_0\not=0$ then $S'(0)q_0\not=0.$ Otherwise we would have
$\Omega'_\xi(T,0,x_0(0))A_{n-1}q_0=0,$  where $\Omega'_{\xi}$ is
nonsingular being a fundamental matrix to (\ref{ls}) (see
\cite{kraop}, Theorem 2.1). But this means that  $A_{n-1}q_0=0$
contradicting the definition of $A_{n-1}.$ Denoting by $z_0$ the
$T$-periodic eigenfunction of (\ref{ss}) satisfying (\ref{orient})
from (\ref{zvio}) we obtain
$$
 \left(z_1(t),...,z_{n-1}(t),{z_0}(t)\right)^*S'(0)q_0=\left(\begin{array}{c}
 0\\ \vdots \\ 0\\ a\end{array}\right)
$$
with $a\not=0.$ From this formula we have
$$
S'(0)q_0=\left(\left(z_1(t),...,z_{n-1}(t),{z_0}(t)\right)^*\right)^{-1}\left(\begin{array}{c}
 0\\ \vdots \\ 0\\ a\end{array}\right)
$$
and by Lemma~\ref{EI} we have $S'(0)q_0=a\dot x_0(0)$ contradicting
Lemma~\ref{conv_lem2}.
\end{proof}

\vskip0.5truecm By combining Theorem~\ref{conv_thm1} and
Lemma~\ref{conv_lem3} we have the following result.

\begin{corollary}\label{conv_cor6}
Assume all the conditions of Theorem~\ref{conv_thm1}. Assume that
the linearized system (\ref{ls}) has $n$ linearly independent
eigenfunctions. Assume that $M^\bot_z(0)=0$ for any not $T$-periodic
eigenfunction $z$ of the adjoint system (\ref{ss}). Then
$$
  \lim_{\epsi\to 0}\frac{\left\|x_\varepsilon(\Delta_\varepsilon)-{x}_0(0)\right\|}{\varepsilon}=0.
$$
\end{corollary}

\begin{proof} By contradiction assume that there exist
$\{\varepsilon_k\}_{k\in\mathbb{N}}\subset(0,1],$ $\epsi_k\to 0$
as $k\to \infty$  and $c_*>0$ such that
\begin{equation}\label{cor6}
  \frac{\left\|x_{\varepsilon_k}(\Delta_{\varepsilon_k})-{x}_0(0)\right\|}{\varepsilon_k}\ge c_*.
\end{equation}
From (\ref{cor6}) we have that the assumptions of
Lemma~\ref{conv_lem3} are satisfied, thus there exists a not
$T$-periodic eigenfunction $z$ of (\ref{ss}) such that
(\ref{lemma3}) holds, but this contradicts (\ref{mainprop1}). The
proof is complete.
\end{proof}

\vskip0.5truecm

Corollaries \ref{conv_cor3} and \ref{conv_cor6} allow us to
formulate the following result.

\begin{corollary}\label{conv_cor7} Assume all the conditions of
Theorem~\ref{conv_thm1}. Assume that system (\ref{ss}) has $n$
linearly independent eigenfunctions and let $\varepsilon_0>0$
sufficiently small. Then there exists a not $T$-periodic
eigenfunction $z$ of (\ref{ss}) such that either
$$
\cos\angle\left(z(0), x_\varepsilon(\Delta_\varepsilon)-
  {x}_0(0)
  \right)\not=0\quad\mbox{ for any }\;\varepsilon\in(0,\varepsilon_0]
$$
or
$$
\lim_{\epsi\to 0}\frac{
\left\|x_\varepsilon(\Delta_\varepsilon)-{x}_0(0)\right\|}{\varepsilon}=0.
$$
\end{corollary}

Finally, observe that from Remark~\ref{LOUD} it follows that under
Malkin's or Loud's assumptions Corollaries~\ref{conv_cor3},
\ref{conv_cor4}, \ref{conv_cor5}, \ref{conv_cor6}, \ref{conv_cor7}
hold true with $\Delta_\epsi=0.$

\medskip

Received February 15, 2007; revised June 2007.

\medskip

\end{document}